\documentclass{article}

\usepackage{fullpage}
\usepackage{amsmath}
\usepackage{amsthm}
\usepackage{amssymb}

\newtheorem{thm}{Theorem}[section]
\newtheorem{prop}[thm]{Proposition}

\newtheorem{lemma}[thm]{Lemma}

\theoremstyle{definition}

\newcommand{\RR}{\mathbf R}
\newcommand{\ZZ}{\mathbf Z}

\title{Nonconvexity of the set of hypergraph degree sequences}
\author{Ricky Ini Liu\\University of Michigan\\Ann Arbor, Michigan\\\texttt{riliu@umich.edu}}

\begin{document}
\maketitle
\begin{abstract}
It is well known that the set of possible degree sequences for a graph on $n$ vertices is the intersection of a lattice and a convex polytope. We show that the set of possible degree sequences for a $k$-uniform hypergraph on $n$ vertices is not the intersection of a lattice and a convex polytope for $k \geq 3$ and $n \geq k+13$. We also show an analogous nonconvexity result for the set of degree sequences of $k$-partite $k$-uniform hypergraphs and the generalized notion of $\lambda$-balanced $k$-uniform hypergraphs.
\end{abstract}

\section{Introduction}
The \emph{degree sequence} of a graph $G$ on vertices $v_1, v_2, \dots, v_n$ is the sequence $d(G)=(d_1, d_2, \dots, d_n)$, where $d_i$ is the degree of the vertex $v_i$ in $G$. The Erd\H{o}s-Gallai Theorem~\cite{ErdosGallai} states that a sequence $(d_1, d_2, \dots, d_n)$ is the degree sequence of a (simple) graph if and only if $\sum_i d_i$ is even and the $d_i$ satisfy a certain set of inequalities. Koren~\cite{Koren} showed that these inequalities define a convex polytope $D_n(2)$, so that the sequences with even sum lying in this polytope are exactly the degree sequences of graphs on $n$ vertices. (For more on this polytope, see \cite{Stanley}.)

We consider the analogous question for $k$-uniform hypergraphs when $k>2$. Klivans and Reiner~\cite{KlivansReiner} verified computationally that the set of degree sequences for $k$-uniform hypergraphs is the intersection of a lattice and a convex polytope for $k=3$ and $n \leq 8$ and asked whether this holds in general. We will show in Section 2 that it does not hold for $k \geq 3$ and $n \geq k+13$.

Similarly, we can associate to a bipartite graph a pair of degree sequences giving the degrees of the vertices in each part. The Gale-Ryser Theorem~\cite{Ryser} gives necessary and sufficient conditions in the form of a system of linear inequalities for a pair of degree sequences to arise from a bipartite graph, so that the set of these pairs of degree sequences can again be described as the intersection of a lattice and a convex polytope. We will show in Section 3 that the analogous result does not hold for $k$-partite $k$-uniform hypergraphs if there exist three parts of sizes at least 5, 6, and 6, respectively. We also generalize the notion of $k$-partite $k$-uniform hypergraphs to that of $\lambda$-balanced $k$-uniform hypergraphs and prove a similar statement in this case.

\section{Hypergraph degree sequences}
A \emph{(simple) $k$-uniform hypergraph} $K$ on the set $[n]=\{1, 2, \dots, n\}$ is a collection of distinct elements (called \emph{hyperedges}) of $\binom{[n]}{k}$, the $k$-element subsets of $[n]$. The \emph{degree sequence} of $K$ is $d(K)=(d_1, d_2, \dots, d_n)$, where $d_i$ is the number of hyperedges in $K$ containing $i$.

We consider degree sequences as points in $\RR^n$. Let $e_i$ be the $i$th standard basis vector, and for any $S = \{i_1, \dots, i_k\} \subset [n]$, write $e_S=e_{i_1i_2\dotsm i_k}=e_{i_1}+e_{i_2}+\dotsb+e_{i_k}$. Each degree sequence $d(K)$ is the sum of some subset of the $e_S$'s, so the convex hull of all such degree sequences is the zonotope
\[D=D_n(k)=\Big\{ \sum_{S \in \binom{[n]}{k}} c_Se_S \mid 0 \leq c_S \leq 1\Big\}.\]
(For more on this polytope, see \cite{BhanuMurthySrinivasan}.)
Moreover, if we let $L \subset \ZZ^n$ be the lattice generated by the $e_S$ consisting of lattice points whose coordinates have sum divisible by $k$ (as long as $n>k$), then each $d(K)$ lies in $D \cap L$. Our main result will be to show that $D \cap L$ contains a point that is not the degree sequence of a $k$-uniform hypergraph when $k \geq 3$.

As a remark, this is closely related to the weaker question of whether every point of $L$ lying in the real cone generated by the $e_S$ lies in the semigroup generated by the $e_S$. This is well known to be the case and is equivalent to normality of the monomial algebra generated by the $\mathbf x^S=x_{i_1}x_{i_2}\cdots x_{i_k}$. (See, for instance, \cite{Sturmfels}.) It is also easy to derive the affirmative answer to this question for $\lambda$-balanced hypergraphs as defined in the next section. The essential difference with the present question is that here we are restricted to using each hyperedge at most once.

For a hypergraph $K$, we will define $D(K)$ to be the zonotope generated by the hyperedges in $K$, so
\[D(K) = \left\{\sum_{S \in K} c_Se_S \mid 0 \leq c_S \leq 1 \right\}.\]

\begin{lemma} \label{face}
Let $K$ be a $k$-uniform hypergraph on $n$ vertices. Then any nonempty face of $D(K)$ is a translate of $D(K^0)$ for some $K^0 \subset K$. Moreover, $D(K) \cap L$ contains a point that is not the degree sequence of a subhypergraph of $K$ if and only if $D(K^0) \cap L$ contains a point that is not the degree sequence of a subhypergraph of $K^0$.
\end{lemma}

\begin{proof}
Choose any weight vector $w \in (\RR^*)^n$. To maximize $w\left(\sum c_Se_S\right)=\sum (c_S \cdot w(e_S))$ for $0 \leq c_S \leq 1$, we must take $c_S = 1$ when $w(e_S)>0$ and $c_S=0$ when $w(e_S)<0$, while $c_S$ can be arbitrary if $w(e_S)=0$. Thus the face on which $w$ is maximized is a translate of $D(K^0)$ by $\sum_{S \in K^+} e_S \in L$, where $K^+$ is the set of hyperedges $S$ on which $w$ is positive. The same argument gives the result for degree sequences (simply restrict $c_S$ to be 0 or 1).
\end{proof}

Therefore it suffices to exhibit a weight vector $w$ to maximize and a point of $D(K^0) \cap L$ that is not the degree sequence of a subhypergraph of $K^0$.

\begin{prop} \label{example1}
Let $k=3$ and $n=16$. If
\[w=(8,6,6,4,1,1,0,0,0,0,-2,-2,-3,-3,-5,-12),\]
then
\[p=(2,1,1,2,1,1,1,1,1,1,1,1,2,2,2,1)\] lies in $D(K^0)\cap L$ but is not the degree sequence of a subhypergraph of $K^0$.
\end{prop}
\begin{proof}
Since the sum of the entries of $p$ is $21 = 3\cdot  7$, $p$ lies in $L$. Also,
\begin{multline*}
p = \frac 13 (e_{2,3,16}+e_{4,5,15}+e_{4,6,15}+e_{5,6,11}+e_{5,6,12}+e_{7,8,9}+e_{7,8,10}+e_{7,9,10}+e_{8,9,10})\\
+\frac 23 (e_{1,4,16}+e_{1,13,15}+e_{1,14,15}+e_{2,13,14}+e_{3,13,14}+e_{4,11,12}).
\end{multline*}
Since $w$ vanishes on each $e_S$ on the right side, it follows that $p \in D(K^0)$. However, $p$ is not the degree sequence of a subhypergraph of $K^0$: since $w_7=w_8=w_9=w_{10}=0$ and otherwise $w_i\neq-w_j$, we have $(e_7+e_8+e_9+e_{10})\cdot e_S$ is 0 or 3 for any $S \in K^0$. But $(e_7+e_8+e_9+e_{10})\cdot p = 4$, which is not divisible by 3, so it cannot be the sum of some $e_S$ for $S \in K^0$.
\end{proof}

Using this, we can easily derive the following.

\begin{thm} \label{main}
For $k \geq 3$ and $n \geq k+13$, the set of degree sequences of $k$-uniform hypergraphs on $n$ vertices is not the intersection of a lattice and a convex polytope.
\end{thm}
\begin{proof}
It suffices to show that there is a point in $D \cap L$ that is not a degree sequence (since $D$ and $L$ are the smallest convex polytope and lattice containing all degree sequences). Combining Lemma~\ref{face} and Proposition~\ref{example1} gives the result for $k=3$ and $n=16$. Since $D_n(k)$ is the face of $D_{n+1}(k)$ with last coordinate 0, Lemma~\ref{face} also gives the result for $k=3$ and $n\geq 16$.

Consider the map $f\colon (d_1, d_2,\dots, d_n) \mapsto (d_1,d_2,\dots,d_n,\frac{1}{k}(d_1+\dots+d_n))$. Then $d$ is a $k$-uniform hypergraph degree sequence on $n$ vertices if and only if $f(d)$ is a $(k+1)$-uniform hypergraph degree sequence on $n+1$ vertices (simply add vertex $n+1$ to all hyperedges). Since $f$ is linear, it also sends $D_n(k)$ into $D_{n+1}(k+1)$, so any counterexample for $(n,k)$ yields a counterexample for $(n+1,k+1)$. An easy induction completes the proof.
\end{proof}

It is possible that with additional work or computation the constant 13 may be improved.

In the next section, we will prove an analogous result for $k$-partite $k$-uniform hypergraphs as well as the more general $\lambda$-balanced hypergraphs. (Our construction below can also be used to prove Theorem~\ref{main} but with a constant of 14 instead of 13.)

\section{$\lambda$-balanced hypergraphs}

Let $\lambda=(\lambda_1, \lambda_2, \dots, \lambda_p)$ be a partition of $k$. We say a $k$-uniform hypergraph is \emph{$\lambda$-balanced} if its vertex set can be partitioned into $p$ sets $V_1, \dots, V_p$ such that each hyperedge contains $\lambda_i$ vertices from $V_i$. (We will also call a hyperedge $\lambda$-balanced if it satisfies this property.) A $(1,1,\dots, 1)$-balanced partition is called \emph{$k$-partite}. Note that every $k$-uniform hypergraph is $(k)$-balanced.

Let $n_i=|V_i|$, and label the vertices in $V_i$ by $v^i_1, v^i_2, \dots, v^i_{n_i}$. We then associate to a $\lambda$-balanced hypergraph $K$ a degree sequence \[d=(d^1_1, d^1_2, d^1_3, \dots;\quad d^2_1, d^2_2, \dots;\quad \dots;\quad d^p_1, d^p_2, \dots),\]
where $d^i_j$ gives the number of hyperedges in $K$ containing vertex $v^i_j$. As before, this degree sequence is $\sum_{S \in K} e_S$, where $e_S$ is the sum of the standard basis vectors in $\RR^{n_1}\times \RR^{n_2} \times \dots \times \RR^{n_p}$ corresponding to vertices in the hyperedge $S$. When $n_i>\lambda_i$ for all $i$, the lattice $L$ generated by all possible $e_S$ consists of all sequences $d$ for which there exists $q\in \ZZ$ such that $\sum_{j=1}^{n_i} d^i_j = \lambda_iq$ for all $i$. (In other words, the sum of the degrees of the vertices in $V_i$ must be the same integer multiple of $\lambda_i$.)

As before, we let $D$ be the zonotope generated by all $e_S$ for $\lambda$-balanced hyperedges $S$ and ask whether all points in $D \cap L$ are degree sequences for $\lambda$-balanced hypergraphs. We will again find that this is not the case for any $\lambda$ when $k\geq 3$ and the $n_i$ are sufficiently large. We first consider a special case.

\begin{prop} \label{example2}
Let $\lambda=(1,1,1)$ and $(n_1, n_2, n_3)=(5,6,6)$. Also let
\[w = (-7,-7,-7,-7,-7;\quad 1, 1, 2, 2, 3, 3;\quad 6, 6,5,5,4,4)\]
and define $K_0$ as in Lemma~\ref{face}. Then 
\[p=(11,9,6,3,1;\quad 2,4,6,8,3,7;\quad 2,4,6,8,3,7)\]
lies in $D(K_0) \cap L$ but is not the degree sequence of a subhypergraph of $K^0$.
\end{prop}
\begin{proof}
Define points
\begin{alignat*}{2}
p^-&=(10,8,4,2,0;\quad&1,3,5,7,2,6;\quad&1,3,5,7,2,6),\\
p^+&=(12,10,8,4,2;\quad&3,5,7,9,4,8;\quad&3,5,7,9,4,8),
\end{alignat*}
so $p=\frac12(p^-+p^+)$.
Note that the sum of the coordinates of the three parts of $p^-$ are all 24, so $p^- \in L$. Likewise, $p^+$ and $p$ also lie in $L$.

Let
\[A=(a_{rs})=
\begin{pmatrix}
1&2&0&0&0&0\\
2&3&0&0&0&0\\
0&0&3&4&0&0\\
0&0&4&5&0&0\\
0&0&0&0&1&3\\
0&0&0&0&3&5\end{pmatrix}.\]
Note that $K^0 = \{\{v^1_q,v^2_r,v^3_s\} \mid 1 \leq q \leq 5, 1\leq r,s \leq 6,  a_{rs} \neq 0\}$.

Then $p^- = \sum e_S$, where the sum ranges over all $S=\{v^1_q, v^2_r, v^3_s\}$ such that $q < a_{rs}$. Likewise $p^+ = \sum e_S$, where instead $q \leq a_{rs}$. Therefore $p^-$, $p^+$, and their midpoint $p$ lie in $D(K_0)$.

We will now show that $p$ is not the degree sequence of a hypergraph that uses only hyperedges in $K^0$. Suppose it were, so that we could write $p=\sum_{S \in K} e_S$ for some $K \subset K_0$. Let $B=(b_{rs})$ be the $6 \times 6$ matrix such that $b_{rs}$ counts the number of $q$ for which $\{v^1_q,v^2_r,v^3_s\} \in K$. Then the sequence of row and column sums of $B$ must both be $(2,4,6,8,3,7)$. Since we also know that $0 \leq b_{rs}\leq 5$, this means that:
\begin{align*}
B_1=\begin{pmatrix}
b_{11}&b_{12}\\b_{21}&b_{22}
\end{pmatrix}
&\in
\left\{
\begin{pmatrix}0&2\\2&2\end{pmatrix},
\begin{pmatrix}1&1\\1&3\end{pmatrix},
\begin{pmatrix}2&0\\0&4\end{pmatrix}
\right\}\\
B_2=\begin{pmatrix}
b_{33}&b_{34}\\b_{43}&b_{44}
\end{pmatrix}
&\in
\left\{
\begin{pmatrix}1&5\\5&3\end{pmatrix},
\begin{pmatrix}2&4\\4&4\end{pmatrix},
\begin{pmatrix}3&3\\3&5\end{pmatrix}
\right\}\\
B_3=\begin{pmatrix}
b_{55}&b_{56}\\b_{65}&b_{66}
\end{pmatrix}
&\in
\left\{
\begin{pmatrix}0&3\\3&4\end{pmatrix},
\begin{pmatrix}1&2\\2&5\end{pmatrix}
\right\}.
\end{align*}

Moreover, for $1 \leq r,s \leq 6$, the pair $\{v^2_r,v^3_s\}$ can appear in at most $\min \{q, b_{rs}\}$ hyperedges with one of the vertices in $\{v^1_1, \dots, v^1_q\}$. Therefore, if we let $\mu=(11,9,6,3,1)$, then $\mu_1+\dots+\mu_q \leq \sum_{r,s} \min \{q, b_{rs}\}$. In other words, if $\nu = (\nu_1, \dots, \nu_5)$ is the partition such that $\nu_q$ counts the number of $b_{rs}$ that are at least $q$, then $\mu_1 + \dots + \mu_q \leq \nu_1 + \dots + \nu_q$.

It is now straightforward to show that there are no possible choices of $B_1$, $B_2$, and $B_3$ satisfying these conditions: if
$B_3 = (\begin{smallmatrix}0&3\\3&4\end{smallmatrix})$, we cannot choose $B_1$ such that both $\mu_1 \leq \nu_1$ and $\mu_1+\mu_2 \leq \nu_1+\nu_2$. Similarly if $B_3=(\begin{smallmatrix}1&2\\2&5\end{smallmatrix})$, we cannot choose $B_2$ such that both $\mu_1+\mu_2+\mu_3 \leq \nu_1+\nu_2+\nu_3$ and $\mu_1+\mu_2+\mu_3+\mu_4 \leq \nu_1+\nu_2+\nu_3+\nu_4$. Thus $p \in D(K_0) \cap L$ is not the degree sequence of a hypergraph using only hyperedges in $K_0$.
\end{proof}

Combining Lemma~\ref{face} and Proposition~\ref{example2} gives our desired result for $3$-partite $3$-uniform hypergraphs, and we can easily extend this result to $k$-partite $k$-uniform hypergraphs.

\begin{thm} \label{partite}
For $k \geq 3$, consider $k$-partite $k$-uniform hypergraphs with parts of sizes $n_1, n_2, \dots, n_k$ for which $n_1 \geq 5$, $n_2 \geq 6$, $n_3 \geq 6$, and $n_i \geq 1$ otherwise. The corresponding set of degree sequences is not the intersection of a lattice and a convex polytope.
\end{thm}
\begin{proof}
As in Theorem~\ref{main}, combining Lemma~\ref{face} and Proposition~\ref{example2} gives the result for $k=3$ and $(n_1, n_2, n_3)=(5,6,6)$. Also note that the polytopes and lattices for $k\geq 3$ with $(n_1, n_2, n_3, n_4, \dots, n_k) = (5,6,6,1,\dots, 1)$ are all identical to the $k=3$ case (by projecting away the last $k-3$ coordinates) so this also proves those cases. Finally, increasing any $n_i$ but restricting to the face of the zonotope where the new vertices have degree 0 again reduces to the same case by Lemma~\ref{face}, completing the proof.
\end{proof}

Theorem~\ref{partite} is also easy to extend to $\lambda$-balanced hypergraphs for all $\lambda$ when $k \geq 3$. Consider a $\lambda$-balanced hypergraph on vertex sets $V_1, \dots, V_p$ of sizes $n_1, n_2, \dots, n_p$. We will say that $(n_1, \dots, n_p)$ is a \emph{$\lambda$-coarsening} of $(m_1, \dots, m_k)$ if each $V_i$ can be partitioned into $\lambda_i$ sets such that the sizes of all the resulting sets are $m_1, \dots, m_k$.

\begin{thm}\label{balanced}
Consider $\lambda$-balanced hypergraphs with parts of sizes $n_1, \dots, n_p$, where $(n_1, \dots, n_p)$ is a $\lambda$-coarsening of $(m_1, m_2, \dots, m_k)$ such that Theorem~\ref{partite} holds for parts of sizes $m_1, \dots, m_k$. (In particular, this will hold whenever the $n_i$ are sufficiently large.) Then the corresponding set of degree sequences is not the intersection of a lattice and a convex polytope.
\end{thm}
\begin{proof}
Let the vertex sets $V_1, \dots, V_p$ have corresponding coarsening $W_1, \dots, W_k$. It suffices to exhibit a weight vector $w$ such that the corresponding $K_0$ as in Lemma~\ref{face} is the complete $k$-partite $k$-uniform hypergraph on $W_1, \dots, W_k$. Indeed, any hyperedge in $K_0$ will be $\lambda$-balanced by the definition of $\lambda$-coarsening, and the lattice generated by hyperedges in $K_0$ is a sublattice of the lattice generated by all $\lambda$-balanced hyperedges. Therefore any counterexample for $K_0$ will yield a counterexample for $\lambda$-balanced hypergraphs as in Lemma~\ref{face}.

To exhibit such a weight vector, let $N$ be an integer larger than any $m_i$. Then let the weight of vertices in $W_1$ be $-(1+N+N^2+\dots+N^{k-2})$ and in $W_i$ be $N^{i-2}$ for $2 \leq i \leq k$. Then the only way to pick $k$ vertices the sum of whose weights is 0 is to take one from each $W_i$. In other words, the only hyperedges in $K_0$ are those that have one vertex from each $W_i$, as desired.
\end{proof}

\section{Acknowledgments}

The author would like to thank Victor Reiner for suggesting this direction of study, as well as for useful discussions and overall encouragement. This work was supported by a National Science Foundation Mathematical Sciences Postdoctoral Research Fellowship.

\bibliography{zonotope}
\bibliographystyle{plain}

\end{document}